\documentclass[12pt]{amsart}
\usepackage[all]{xy}
\usepackage{amssymb}
\usepackage{amsthm}
\usepackage{amsmath}
\usepackage{amscd,enumitem}
\usepackage{verbatim}
\usepackage{eurosym}
\usepackage{graphicx}
\usepackage{dsfont}
\usepackage{stmaryrd}
\usepackage{color}
\usepackage{float}
\usepackage{longtable}
\usepackage{dcolumn}
\usepackage[mathscr]{eucal}
\usepackage[all]{xy}
\usepackage{hyperref}
\usepackage[title]{appendix}
\usepackage[usenames,dvipsnames]{xcolor}
\usepackage{bbm}
\usepackage[textheight=8.85in, textwidth=6.8in]{geometry}
\usepackage{multirow}
\usepackage{caption}
\usepackage{MnSymbol}
\theoremstyle{plain}
\newtheorem{theorem}{Theorem}[section]

\newtheorem{lemma}[theorem]{Lemma}

\theoremstyle{definition}

\theoremstyle{remark}
\newtheorem*{remark}{Remark}

\newcommand{\F}{\mathbb{F}}

\newcommand{\C}{\mathbb{C}}

\catcode`,\active

\catcode`\,12

\newcommand{\Fq}{\mathbb{F}_q}

\newcommand{\Fqh}{\widehat{\mathbb{F}}_q^\times}    

\newcommand{\n}{\nu}
\newcommand{\on}{\overline{\nu}}

\numberwithin{equation}{section}

\begin{document}
	\title[On Matrices Arising in Finite Field Hypergeometric Functions]{On Matrices Arising in Finite Field Hypergeometric Functions}

	\author{Satoshi Kumabe\and Hasan Saad}
		\address[Satoshi Kumabe]{Joint Graduate School of Mathematics for Innovation, Kyushu University, 744, Motooka, Nishi-ku, Fukuoka, 819-0395, Japan}
		
	\email{kuma511ssk@gmail.com}

	\address[Hasan Saad]{Department of Mathematics, University of Virginia, Charlottesville, VA 22904, USA}
	
	\email{hs7gy@virginia.edu}

	\begin{abstract}
	In \cite{Lehmer56, Lehmer60}, Lehmer constructs four classes of matrices constructed from roots of unity for which the characteristic polynomials and the $k$-th powers can be determined explicitly. Here we study a class of matrices which arise naturally in transformation formulas of finite field hypergeometric functions and whose entries are roots of unity and zeroes. We determine the characteristic polynomial, eigenvalues, eigenvectors, and $k$-th powers of these matrices. In particular, the eigenvalues are natural systems of products of Jacobi sums.
	\end{abstract}
	\maketitle
	
	\section{Introduction}

	In \cite{Lehmer56}, Lehmer remarks that the class of matrices for which one can explicitly determine the eigenvalues and the general $k$-th power is very limited. Using the Legendre character on finite fields, Lehmer constructs two classes of matrices for which this is possible. More generally, using characters of arbitrary orders, Carlitz \cite{Carlitz} and Lehmer \cite{Lehmer60} construct other classes of matrices for which they determine the characteristic polynomials and $k$-th powers.
	
	Here we consider a class of matrices, whose entries are roots of unity and zeroes, which arise in the transformation formulas for Gaussian hypergeometric functions over finite fields defined by Greene\cite{Greene}. We first recall the definition of these functions. If $p$ is a prime, $q=p^r, n\geq 1,$ and $A_1,\ldots,A_n,B_2,\ldots,B_{n}$ are complex-valued multiplicative characters over $\F_q^\times,$ then the finite field hypergeometric functions are defined by 
	\begin{align}\label{Greene_Def}
	{_{n}F_{n-1}}\left(\begin{array}{cccc}
		A_1, & A_2,& \ldots, & A_n\\
		~& B_2, &\ldots, &B_{n}
	\end{array}\mid x\right)_q:=\frac{q}{q-1}\sum_{\chi}{A_1\chi\choose\chi}{A_2\chi\choose B_2\chi}\cdots{A_n\chi\choose B_{n}\chi}\chi(x),
	\end{align}
	where the summation is over multiplicative characters $\chi$ of $\F_q^\times,$ and where the binomial coefficient $\binom{A}{B}$ is a normalized Jacobi sum, given by
	\begin{align}\label{binomial}
	{A\choose B}:=\frac{B(-1)}{q}J(A,\overline{B}):=\frac{B(-1)}{q}\sum\limits_{x\in\F_q}A(x)\overline{B}(1-x).
	\end{align}
	These functions have deep connections to \'etale cohomology \cite{Katz_exponential_sum} and often arise in geometry where they count the number of $\F_q$-points on various algebraic varieties (see Theorem 1.5 of \cite{beukers_cohen_mellit}). For example, if $\lambda\in\F_q\setminus\{0,1\}$ and $E_\lambda$ is the Legendre normal form elliptic curve
	$$
	E_\lambda:\ \ \ y^2 = x(x-1)(x-\lambda),
	$$
	then (see \cite[Section 4]{Koike} and \cite[Theorem 1]{Ono}) we have that
	$$
	\# E_\lambda(\F_q) = 1+q+q\cdot\phi_q(-1)\cdot {_{2}F_{1}}\left(\begin{array}{ccc}
		\phi_q, & \phi_q\\
		~& \varepsilon
	\end{array}\mid \lambda\right)_q,
	$$
	where $\phi_q$ and $\varepsilon$ are respectively the Legendre symbol and the trivial character on $\F_q^\times.$

	Moreover, these functions satisfy analogues of several transformation formulas of their classical counterparts, such as the generalized Euler integral transform (see \cite[(4.1.1)]{Slater}). More precisely, we have (see \cite[Theorem 3.13]{Greene}) that 
	\begin{align}\label{TransFormula}
		&{_{n+1}F_{n}}\left(\begin{array}{ccccc}
			A_1, & A_2, &\ldots, & A_n, & A_{n+1}\\
			~&{B_2}, &\ldots, &{B_{n}} & {B_{n+1}}
		\end{array}\mid x\right)_{{q}} \\ 
		& = {\frac{A_{n+1}B_{n+1}(-1)}{q}}\sum\limits_{y\in\F_q}{_{n}F_{n-1}}\left(\begin{array}{cccc}
			A_1, & A_2, &\ldots, & A_n\\
			~& {B_2}, &\ldots, & {B_{n}}
		\end{array}\mid xy\right)_{{q}}\cdot{ A_{n+1}(y)\overline{A_{n+1}}B_{n+1}(1-y)}.\notag
	\end{align}

	Motivated by the transformation formula (\ref{TransFormula}), Ono and Griffin--Rolen study the matrix corresponding to this transformation when $q=p^r$ is odd, $A_{n+1}=\phi_q,$ and $B_{n+1}=\epsilon$. More precisely, consider the $(q-2)\times(q-2)$ matrix $M=(M_{ij})$ indexed by $i,j\in\F_q\setminus\{0,1\},$ where
	$$
	M_{ij}=\phi_q(1-ij)\phi_q(ij)
	$$
	and let $f_q$ be its characteristic polynomial.
	In this notation, Griffin and Rolen prove \cite{Griffin_Rolen}  a conjecture by Ono that
	\begin{equation}\label{GF-Equation}
		f_q(x)=\begin{cases}
			(x+1)(x-1)(x+2)(x^2-q)^{(q-5)/2} & \text{ if } \phi_q(-1)=1,\\
			x(x^2-3)(x^2-q)^{(q-5)/2} & \text{ if }\phi_q(-1)=-1.
		\end{cases}
	\end{equation}
	The purpose of this paper is to study, à la Lehmer, a more general analogue of the matrix $M$ that arises when the characters $A_{n+1}$ and $B_{n+1}$ are arbitrary. More precisely, we consider the $({q}-1)\times({q}-1)$ matrix $M_q=(M_q)_{ij}$ indexed by $i,j\in\F_q^\times,$ where
	$$
	(M_q)_{ij}:= A(ij)\overline{A}B(1-ij).
	$$
	We first determine the characteristic polynomial $f_q$ of $M_q.$ 
	
	\begin{theorem}\label{Thm_char_poly}
		If $p$ is an odd prime, $q=p^r,$ and $\omega$ is a character of order $q-1$ of $\F_q^\times$, then
		\begin{align*}
			f_q(x) = (x-J(\overline{A}B, A))(x - J(\overline{A}B,\overline{A}\phi)) \prod_{l = 1}^{\frac{q-3}{2}}(x^2 - J(\overline{A}B, A\omega^l)J(\overline{A}B, A\overline{\omega^l})).
		\end{align*}
	\end{theorem}
	
		\begin{remark}
			Our proof explicitly determines the eigenvectors of $M_q.$
		\end{remark}
	
	Furthermore, when $B=\varepsilon$ and $k\geq 1,$ we explicitly determine the entries of $M_q^k.$ 
	
	\begin{theorem}\label{Thm_k-th_power_Mq}
		If $k\geq 1,$ we write $k=2l$ if $k$ is even and $k=2l+1$ if $k$ is odd. In this notation, if $p$ is an odd prime, $q=p^r,$ and $B=\varepsilon,$ then we have
			$$
			(M_q^k)_{ij}=A^l(-1)\cdot q^{k-1}\cdot{_kF_{k-1}}\left(\begin{array}{cccc}
				A_1, & A_2,& \ldots, & A_k\\
				~& B_2, &\ldots, &B_k
			\end{array}\mid \frac{j^{(-1)^k}}{i}\right)_q,
			$$
			where 
			$$
			A_n	= \begin{cases}
				A & \ \ \ \text{ if } 1\leq n \leq l \\
				\varepsilon & \ \ \ \text{otherwise}
			\end{cases}\ \ \ \ \ and \ \ \ \ \ 
			B_n	= \begin{cases}
				\varepsilon & \ \ \ \text{ if } 2\leq n \leq l \\
				\overline{A} & \ \ \ \text{otherwise}.
			\end{cases}
			$$
	\end{theorem}

		\begin{remark}
			If $B\neq\varepsilon,$ the entries of $M_q^k$ can be written in terms of more general finite field hypergeometric functions, such as the ones by McCarthy {{\cite[Definition 2.4]{McCarthy_gen_def}}} and Otsubo {{\cite[Definition 2.7]{Otsubo}}}. The proof is analogous to the proof of Theorem~\ref{Thm_k-th_power_Mq}.
		\end{remark}
	
	The paper is organized as follows. In Section~\ref{NutsBolts-Section}, we recall facts concerning characters, finite field hypergeometric functions, and determine the action of $M_q$ on an appropriate basis. In Section~\ref{Proofs-Section}, we prove Theorems~\ref{Thm_char_poly} and \ref{Thm_k-th_power_Mq}.
	
			\section*{Acknowledgements}
	The authors would like to thank Ken Ono for introducing the paper due to Griffin--Rolen to them and for many valuable comments. This paper was written while the first author was visiting the University of Virginia. He would like to deeply thank Ken Ono for his hospitality and support during his visit. 
	The first author was supported by JSPS KAKENHI Grant Number JP22KJ2477 and WISE program (MEXT) at Kyushu University. The second author furthermore thanks Ken Ono for providing research support with the Thomas Jefferson Fund and the NSF Grant (DMS-2002265 and DMS-2055118).
	
	\section{Nuts and Bolts}\label{NutsBolts-Section}
	
	Here we recall facts about characters on finite fields and hypergeometric functions. We also determine the behavior of $M_q$ on an appropriate set of vectors. 
	
	First, we denote by $\Fqh$ the group of characters on $\F_q^\times.$ It is well known (see \cite[Proposition 8.1.2]{Ireland_Rosen}) that if $\chi\in\Fqh,$ then
	\begin{equation}\label{SumValues}
		\sum\limits_{x\in\F_q}\chi(x) = \begin{cases}
			q-1 & \text{\it if }\chi=\varepsilon \\
			0 & \text{\it otherwise }
		\end{cases}
	\end{equation}
	and that if $x\in\F_q,$ then 
	\begin{equation}\label{SumCharacters}
		\sum\limits_{\chi\in\Fqh}\chi(x) = \begin{cases}
			q-1 & \text{\it if }x=1 \\
			0 & \text{\it otherwise.}
		\end{cases}
	\end{equation}
	Furthermore, if $A,B\in\Fqh,$ then the following properties of binomial coefficients are known (\cite[(2.6) through (2.8)]{Greene}]).
	
	\begin{align}\label{binomial_a}
		\binom{A}{B} = \binom{A}{A\overline{B}},
	\end{align}
	
	\begin{align}\label{binomial_b}
		\binom{A}{B} = B(-1)\cdot \binom{B\overline{A}}{B},
	\end{align}
	
	\begin{align}\label{binomial_c}
		\binom{A}{B} = \overline{AB}(-1)\cdot\binom{\overline{B}}{\overline{A}}.
	\end{align}
	
	To state our results, fix a generator $\omega$ of $\Fqh.$ If $1\leq l\leq q-1,$  then we define the vectors $\boldsymbol{w}^l$ indexed by $i\in\F_q^\times,$ where
	$$
	\boldsymbol{w}^{l}_i=\omega^l(i).
	$$
 	The following lemma determines $M_q\boldsymbol{w}^l.$
	\begin{lemma}\label{Lem_Mq_conjugate_vectors}
		If $1\leq l\leq q-1,$ then we have
		\begin{align*}
			M_q \boldsymbol{w}^l=J(\overline{A}B, A\omega^l){\boldsymbol{w}}^{q-1-l}. 
		\end{align*}
	\end{lemma}

	\begin{proof}
		Fix $l.$ Then, for $i\in\F_q^\times,$ we have that
		$$
			(M_q\boldsymbol{w}^l)_i=\sum\limits_{j\in\F_q^\times} A(ij)\overline{A}B(1-ij)\omega^l(j).
		$$
		Replacing $j$ by {$j/i$}, we have that
		\begin{align*}
			(M_q\boldsymbol{w}^l)_i&=\sum\limits_{j\in\F_q^\times} A(j)\overline{A}B(1-j)\omega^l\left(\frac{j}{i}\right) \\
			&=\overline{\omega^l}(i)\sum\limits_{j\in\F_q^\times} (A\omega^l)(j)\overline{A}B(1-j) \\
			&=J(\overline{A}B, A\omega^l){\boldsymbol{w}^{q-1-l}_i}.
		\end{align*}
	\end{proof}

	\begin{remark}
		Recall that the Fourier transform of $f:\F_q\to\C$ is a function $\widehat{f}:\Fqh\to\C$ defined by 
		\begin{align}\label{Fourier_transform_of_f}
			\widehat{f}(\n)=\sum_{\lambda\in \Fq}f(\lambda)\on(\lambda).
		\end{align}
		By a similar argument to the proof of Lemma \ref{Lem_Mq_conjugate_vectors}, we have that the Fourier transforms of the components of $M_q^2$ are products of two Jacobi sums.
\end{remark}

	To determine the quadratic terms in Theorem~\ref{Thm_char_poly}, we make use of the following lemma which follows from a direct computation.
	
	\begin{lemma}\label{LemmaV1V2}
		If $M$ is an $n\times n$ matrix,  $\lambda_1,\lambda_2\in\C,$ and $v_1\neq \pm v_2\in\C^n$ such that
		$$
		Mv_1 = \lambda_1 v_2, Mv_2=\lambda_2v_1,
		$$
		then the vectors $v_1\pm\sqrt{\lambda_1/\lambda_2}v_2$ are eigenvectors of $M$ corresponding to the eigenvalues $\pm\sqrt{\lambda_1\lambda_2}.$
	\end{lemma}
 
	Finally, we need to determine the inverse change-of-basis matrix for the basis $\{\boldsymbol{w}^l\}_{1\leq l\leq q-1}.$
	\begin{lemma}\label{PInverse}
	If $P$ is the matrix given by $P_{ij}=\omega^j(i),$ where $i\in\F_q^\times$ and $1\leq j\leq q-1,$  then we have\footnote{Note that the indices for rows and columns are inverted in $P^{-1}.$ In other words, for $P^{-1},$ $1\leq i\leq q-1$ and $j\in\F_q^\times.$}
	$$
	(P^{-1})_{ij}=\frac{1}{q-1}\overline{\omega^i(j)}.
	$$
	\end{lemma}
	
	\begin{proof}
		Note that
		$$
			\sum\limits_{k\in\F_q^\times}\omega^k(i)\cdot\frac{1}{q-1}\overline{\omega^k(j)}= \frac{1}{q-1}\sum\limits_{k\in\F_q^\times}\omega^k\left(\frac{i}{j}\right).
		$$
		Since $\omega$ is a generator of $\Fqh,$ the lemma follows by (\ref{SumCharacters}).
	\end{proof}

	%%%%%%%%%%%%%new Section%%%%%%%%%%%
	\section{Proofs of Theorems~\ref{Thm_char_poly} and \ref{Thm_k-th_power_Mq}}\label{Proofs-Section}
	We now prove Theorems~\ref{Thm_char_poly} and ~\ref{Thm_k-th_power_Mq}.

	\begin{proof}[Proof of Theorem~\ref{Thm_char_poly}]
	Applying Lemma~\ref{Lem_Mq_conjugate_vectors} with $l = (q-1)/2$ and $l = q-1$ shows that $x-J(\overline{A}\phi,\overline{A})$ and $x-J(\overline{A},A)$ divide $f_q(x).$ Similarly, applying Lemma~\ref{Lem_Mq_conjugate_vectors} with $1\leq l\leq (q-3)/2$ and Lemma~\ref{LemmaV1V2} to the vectors $\boldsymbol{w}^l$ and $\boldsymbol{w}^{q-1-l}$ shows that $x^2-J(\overline{A}B, A\omega^l)J(\overline{A}B,A\overline{\omega^l})$ divides $f_q(x).$
	\end{proof}
	
	\begin{proof}[Proof of Theorem~\ref{Thm_k-th_power_Mq}]
	 We prove this theorem only when $k=2l$ is even. Applying  Lemma~\ref{Lem_Mq_conjugate_vectors} twice, we have that
	$$
	M_q^2 = PDP^{-1},
	$$
	where 
	$$
	D_{mn}=\begin{cases}
		J(\overline{A}, A\omega^m)J(\overline{A}, A\overline{\omega^m}) & \text{\it if \ }m=n \\
		0 & \text{\it otherwise },
	\end{cases}
	$$
	and $P_{ij}=\omega^j(i)$ for $i\in\F_q^\times$ and $1\leq j\leq q-1.$ 
	By Lemma~\ref{PInverse} and a direct computation, we have that
	$$
	(M_q^{2l})_{ij}=\frac{1}{q-1}\sum\limits_{m=1}^{q-1}\omega^m\left(\frac{i}{j}\right)J(\overline{A},A\omega^m)^lJ(\overline{A},A\overline{\omega^m})^l.
	$$
	By applying (\ref{binomial}) and (\ref{binomial_a}) through (\ref{binomial_c}), we have that
	$$
	(M_q^k)_{ij}=A^{l}(-1)\cdot \frac{q^{m}}{q-1}\sum\limits_{m=1}^{q-1}\binom{\overline{\omega^m}}{\overline{A}\overline{\omega^m}}^l\binom{A\overline{\omega^m}}{\overline{\omega^m}}^l \overline{\omega^m}\left(\frac{j}{i}\right).
	$$
	Since $\overline{\omega}$ generates $\Fqh,$ the theorem follows from (\ref{Greene_Def}). The proof is similar when $k$ is odd.
	\end{proof}
	
		%\bibliography{references} %
	%\bibliographystyle{amsalpha} %style of references\\

	\end{document}